\documentclass[draft]{amsart}
\usepackage{hyperref}
\hypersetup{nesting=true,debug=true,naturalnames=true}
\usepackage{graphicx,amssymb,comment}

\usepackage[square,sort&compress,comma,numbers]{natbib}
\usepackage{nicefrac,xcolor,upref}

\usepackage{amscd,amsthm,amsmath,amssymb,amsfonts}

\usepackage{mathrsfs}

\newtheorem{theorem}{Theorem}[section]
\newtheorem{lemma}[theorem]{Lemma} 
\newtheorem{proposition}[theorem]{Proposition} 
\newtheorem{definition}[theorem]{Definition} 
\newtheorem{problem}[theorem]{Problem} 
\newtheorem{corollary}[theorem]{Corollary} 

\numberwithin{equation}{section}

\def\Q{{\mathbb {Q}}}

\def\Z{{\mathbb Z}} \def\F{{\mathbb F}}
\def\R{{\mathbb R}} 
\def\C{{\mathbb C}}  
 
   \def\resp{{\rm resp.,  }}

\def\house#1{\setbox1=\hbox{$\,#1\,$}%
\dimen1=\ht1 \advance\dimen1 by 2pt \dimen2=\dp1 \advance\dimen2 by 2pt
\setbox1=\hbox{\vrule height\dimen1 depth\dimen2\box1\vrule}%
\setbox1=\vbox{\hrule\box1}%
\advance\dimen1 by .4pt \ht1=\dimen1
\advance\dimen2 by .4pt \dp1=\dimen2 \box1\relax}

  \def\deg{{\rm deg}}
   
  \def\eps{{\varepsilon}}

\def\build#1_#2^#3{\mathrel{\mathop{\kern 0pt#1}\limits_{#2}^{#3}}}

\def\date {le\ {\the\day}\ \ifcase\month\or
janvier\or fevrier\or mars\or avril\or mai\or juin\or juillet\or
ao\^ut\or septembre\or octobre\or novembre\or
d\'ecembre\fi\ {\oldstyle\the\year}}

\font\fivegoth=eufm5 \font\sevengoth=eufm7 \font\tengoth=eufm10

\newfam\gothfam \scriptscriptfont\gothfam=\fivegoth
\textfont\gothfam=\tengoth \scriptfont\gothfam=\sevengoth

\def\hw{{\widehat w}} 
\def\hla{{\widehat \lambda}}

\def\F{{\mathbb F}}

\begin{document}

\vskip 2mm

\title{Mahler's and Koksma's classifications in fields of power series}
 
\vskip 8mm

\author{Jason Bell}
\address{Department of Pure Mathematics, University of Waterloo, Waterloo, ON, Canada N2L3G1}
\email{jpbell@uwaterloo.ca}

\author{Yann Bugeaud}
\address{Universit\'e de Strasbourg, Math\'ematiques,
7, rue Ren\'e Descartes, 67084 STRASBOURG  (France)}
\email{bugeaud@math.unistra.fr}

\begin{abstract}
Let $q$ a prime power and $\F_q$ the finite field of $q$ elements. 
We study the analogues of Mahler's and Koksma's classifications of complex numbers for 
power series in $\F_q((T^{-1}))$. Among other results, we establish that both classifications coincide, thereby 
answering a question of Ooto. 
\end{abstract}


\maketitle

\section{Introduction}

Mahler \cite{Mah32}, in 1932, and Koksma \cite{Ko39}, in 1939, 
introduced two related measures for the 
quality of approximation of a complex number $\xi$ by algebraic numbers.
For any integer $n \ge 1$, we denote by
$w_n (\xi)$ the supremum of the real numbers $w$ for which
$$
0 < |P(\xi)| < H(P)^{-w}
$$
has infinitely many solutions in integer polynomials $P(X)$ of
degree at most $n$. Here, $H(P)$ stands for the na\"\i ve height of the
polynomial $P(X)$, that is, the maximum of the absolute values of
its coefficients. Further, we set 
$$
w(\xi) = \limsup_{n \to \infty}  \frac{w_n(\xi)}{n}
$$
and, according to Mahler \cite{Mah32}, we say that $\xi$ is
$$
\displaylines{
\hbox{an $A$-number, if $w(\xi) = 0$;} \cr
\hbox{an $S$-number, if $0 < w(\xi) < \infty$;} \cr
\hbox{a $T$-number, if $w(\xi) = \infty $ and $w_n(\xi) < \infty$, for any
integer $n\ge 1$;} \cr
\hbox{a $U$-number, if $w(\xi) = \infty $ and $w_n(\xi) = \infty$, for some
integer $n\ge 1$.}\cr}
$$
The set of complex $A$-numbers is the set of complex algebraic numbers. 
In the sense of the Lebesgue measure, almost all numbers are $S$-numbers. 
Liouville numbers (which, by definition, are the real numbers $\xi$ such that 
$w_1(\xi)$ is infinite) are examples of $U$-numbers, while the existence of
$T$-numbers remained an open problem during nearly forty years, until it was
confirmed by Schmidt \cite{Schm70,Schm71}.

Following Koksma \cite{Ko39}, for any integer $n \ge 1$, we denote by
$w_n^* (\xi)$ the supremum of the real numbers $w^*$ for which
$$
0 < |\xi - \alpha| < H(\alpha)^{-w^*-1}
$$
has infinitely many solutions in complex algebraic numbers $\alpha$ of
degree at most $n$. Here, $H(\alpha)$ stands for the na\"\i ve height of $\alpha$,
that is, the na\"\i ve height of its minimal defining polynomial over the integers. 
Koksma \cite{Ko39} defined $A^*$-, $S^*$-, $T^*$- and $U^*$-numbers as above, using $w_n^*$
in place of $w_n$. Namely, setting
$$
w^*(\xi) = \limsup_{n \to \infty} \frac{w_n^* (\xi)}{n},
$$
we say that $\xi$ is 
$$
\displaylines{
\hbox{an $A^*$-number, if $w^*(\xi) = 0$;} \cr
\hbox{an $S^*$-number, if $0 < w^*(\xi) < \infty$;} \cr
\hbox{a $T^*$-number, if $w^*(\xi) = \infty $ and $w_n^*(\xi) < \infty$, for any
integer $n\ge 1$;} \cr
\hbox{a $U^*$-number, if $w^*(\xi) = \infty $ and $w_n^* (\xi) = \infty$, for some
integer $n\ge 1$.}\cr}
$$
Koksma proved that this classification of numbers
is equivalent to the Mahler one, in the sense that the classes $A$, $S$, $T$, $U$ coincide 
with the classes $A^*$, $S^*$, $T^*$, $U^*$, respectively. 
For more information on the functions $w_n$ and $w_n^*$, 
the reader is directed to \cite{BuLiv,BuDurham}. 


Likewise, we can divide the sets of real numbers and $p$-adic numbers in classes $A$, $S$, $T$, $U$ and 
$A^*$, $S^*$, $T^*$, $U^*$. 
However, there is a subtle difference with the case of complex numbers,
since the field $\R$ of real numbers and the field $\Q_p$ of $p$-adic numbers are not algebraically closed. 
This means that, in the definition of the exponent $w_n^* (\xi)$ for a real (\resp $p$-adic)
number $\xi$, we have to decide whether the algebraic approximants $\alpha$ 
are to be taken in $\C$ (\resp in an algebraic closure of $\Q_p$) or in $\R$ 
(\resp in $\Q_p$). Fortunately, in both cases, it makes no difference, 
as shown in \cite{Bu03,BuLiv}. 
For instance, it has been proved that, if there is $\alpha$ of degree $n$ in 
an algebraic closure of $\Q_p$ satisfying $| \xi - \alpha| < H(\alpha)^{-1-w^*}$, then there exists $\alpha'$ in 
$\Q_p$, algebraic of degree at most $n$, such that $H(\alpha') \le c H(\alpha)$ and 
$| \xi - \alpha'| \le c H(\alpha')^{-1-w^*}$, where $c$ depends only on $\xi$ and on $n$. 

The analogous question has not yet been clarified for Diophantine approximation in the
field $\F_q ((T^{-1}))$ of power series over the finite field $\F_q$. Different authors have different 
practices, some of them define $w_n^*$ by restricting to algebraic elements in $\F_q ((T^{-1}))$, while some
others allow algebraic elements to lie in an algebraic closure of $\F_q ((T^{-1}))$. 
One of the aims of the present paper is precisely to clarify this point. 

Our framework is the following. Let $p$ be a prime number and $q = p^f$ an integer 
power of $p$. Any non-zero element $\xi$ in $\F_q ((T^{-1}))$ can be written
$$
\xi = \sum_{n=N}^{+ \infty} \, a_n T^{-n},
$$
where $N$ is in $\Z$, $a_N \not= 0$, and $a_n$ is in $\F_q$ for $n \ge N$. We define 
a valuation $\nu$ and an absolute value $| \cdot |$ 
on $\F_q ((T^{-1}))$ by setting $\nu (\xi) = N$, $| \xi | := q^{-N}$, and $\nu (0) = + \infty$, $|0|:= 0$. 
In particular, if $R(T)$ is a non-zero polynomial in $\F_q[T]$, then we have $|R| = q^{\deg(R)}$. 
The field $\F_q ((T^{-1}))$ is the completion with respect to $\nu$ of the 
quotient field $\F_q(T)$ of the polynomial ring $\F_q [T]$. 
It is not algebraically closed. 
Following \cite{Tha12}, we denote by $C_{\infty}$ the completion of its algebraic closure. 
To describe precisely the set of algebraic elements in $C_{\infty}$ is rather complicated. 
Indeed, Abhyankar \cite{Ab56} pointed out that it contains the element 
$$
T^{-1/p} + T^{-1/p^2} + T^{-1/p^3} + \cdots ,
$$
which is a root of the polynomial $T X^p - T X - 1$. 
Kedlaya \cite{Ked01,Ked17} constructed an algebraic closure of $K((T^{-1}))$ for any field $K$ 
of positive characteristic in terms of certain generalized power series.

There should be no confusion between the variable $T$ and the notion of $T$-number.

The height $H(P)$ of a polynomial $P(X)$ over $\F_q[T]$ is the maximum of the absolute values of its 
coefficients. A power series in $C_\infty$ is called algebraic if it is a root of a nonzero 
polynomial with coefficients in $\F_q[T]$. 
Its height is then the height of its minimal 
defining polynomial over $\F_q[T]$. 
We define the exponents of approximation $w_n$ and $w_n^*$ as follows.

\begin{definition}
\label{Def:1.1}
Let $\xi$ be in $\F_q ((T^{-1}))$. Let $n \ge 1$ be an integer.
We denote by
$w_n (\xi)$ the supremum of the real numbers $w$ for which
$$
0 < |P(\xi )| < H(P)^{-w}
$$
has infinitely many solutions in polynomials $P(X)$ over $\F_q[T]$ of
degree at most $n$. 
We denote by
$w_n^* (\xi)$ the supremum of the real numbers $w^*$ for which
$$
0 < |\xi - \alpha| < H(\alpha)^{-w^*-1}
$$
has infinitely many solutions in algebraic power series $\alpha$ in $\F_q ((T^{-1}))$ of
degree at most $n$.
\end{definition}

An important point in the definition of $w_n^*$ is that we require that the approximants $\alpha$ lie 
in $\F_q ((T^{-1}))$. In the existing literature, it is not always clearly specified 
whether the algebraic approximants are 
taken in $C_{\infty}$ or in $\F_q ((T^{-1}))$. To take this into account, we 
introduce the following exponents of approximation, where we use the superscript ${}^@$ 
to refer to the field $C_{\infty}$.

\begin{definition}
\label{Def:1.3}
Let $\xi$ be in $\F_q ((T^{-1}))$. Let $n \ge 1$ be an integer.
We denote by $w_n^@ (\xi)$ the supremum of the real numbers $w^@$ for which
$$
0 < |\xi - \alpha| < H(\alpha)^{-w^@-1}
$$
has infinitely many solutions in algebraic power series $\alpha$ in $C_{\infty}$ of
degree at most~$n$. 
\end{definition}

Clearly, we have $w_n^@ (\xi) \ge w_n^* (\xi)$ for every $n \ge 1$ and every $\xi$ in $\F_q ((T^{-1}))$. 
The first aim of this paper is to establish that the functions $w_n^*$ and $w_n^@$ coincide. 

\begin{theorem}
\label{Th:2.0}
For any $\xi$ in $\F_q ((T^{-1}))$ and any integer $n \ge 1$, we have
$$
w_n^* (\xi) = w_n^@ (\xi).
$$
\end{theorem}

Theorem \ref{Th:2.0} is not surprising, since it seems 
to be very unlikely that a power series in $\F_q ((T^{-1}))$ 
could be better approximated by 
algebraic power series in $C_{\infty} \setminus \F_q ((T^{-1}))$ than by 
algebraic power series in $\F_q ((T^{-1}))$. Difficulties arise because of the existence 
of polynomials over $\F_q[T]$ which are 
not separable and of the lack of a Rolle Lemma, which is a key ingredient 
for the proof of the analogous result for
the classifications of real and $p$-adic numbers. 

Exactly as Mahler and Koksma did, we divide the set of power series in $\F_q ((T^{-1}))$ in 
classes $A$, $S$, $T$, $U$, $A^*$, $S^*$, $T^*$, and $U^*$, by using the exponents 
of approximation $w_n$ and $w_n^*$. 
It is convenient to keep 
the same terminology and to use $S$-numbers, etc., although we are concerned with power series
and not with `numbers'. This has been done by Bundschuh \cite{Bund78}, who gave some explicit 
examples of $U$-numbers. 
Ooto \cite[p. 145]{Oo17} observed that, by the
currently known results (with $w_n^@$ used instead of $w_n^*$ in the definitions 
of the classes), the sets of $A$-numbers and of $A^*$-numbers
coincide, as do the sets of $U$-numbers and of $U^*$-numbers. Furthermore, an $S$-number is an $S^*$-number, 
while a $T^*$-number is a $T$-number. However, it is not known whether the sets 
of $S$-numbers (\resp $T$-numbers) and of $S^*$-numbers (\resp $T^*$-numbers) coincide. 
The second aim of this paper is to establish that these sets coincide, thereby 
answering \cite[Problem 5.9]{Oo17}.

\begin{theorem}
\label{Th:Tnbs}
In the field $\F_q ((T^{-1}))$ the classes $A$, $S$, $T$, $U$ coincide 
with the classes $A^*$, $S^*$, $T^*$, $U^*$, respectively. 
\end{theorem}

In 2019 Ooto \cite{Oo19} proved the existence of $T^*$-numbers and, consequently, that of 
$T$-numbers. 
His proof is fundamentally different from that of the existence of real 
$T$-numbers by Schmidt \cite{Schm70,Schm71}, whose 
complicated construction rests on a result of Wirsing \cite{Wir71} 
(alternatively, one can use a consequence of 
Schmidt Subspace Theorem) on the approximation to real algebraic numbers by algebraic numbers of lower degree. 
In the power series setting, no analogue of Schmidt Subspace Theorem, or even to Roth Theorem, holds: 
Liouville's result is best possible, as was shown by Mahler \cite{Mah49}. 

Theorem \ref{Th:Tnbs} is an immediate consequence of the following statement. 

\begin{theorem}
\label{Th:wineq}
Let $\xi$ be in $\F_q ((T^{-1}))$ and $n$ be a positive integer. Then, we have
$$
w_n (\xi) - n + 1 \le w_n^* (\xi) \le w_n (\xi). 
$$
\end{theorem}

Theorem \ref{Th:wineq} answers \cite[Problem 5.8]{Oo17} 
and improves \cite[Proposition 5.6]{Oo17}, which asserts that, 
for any positive integer $n$ and any $\xi$ in $\F_q ((T^{-1}))$, we have
$$
\frac{w_n (\xi)}{p^k} - n + \frac{2}{p^k}  - 1 \le w_n^@ (\xi) \le w_n (\xi),
$$
where $k$ is the integer defined by $p^k \le n < p^{k+1}$.

Our next result is, in part, a metric statement. It provides a power series 
analogue to classical statements already established in the real and in the $p$-adic settings. 
Throughout this paper, `almost all' 
always refer to the Haar measure on $\F_q ((T^{-1}))$.

\begin{theorem}
\label{Th:Metric}
For any positive integer $n$ and any $\xi$ in $\F_q ((T^{-1}))$ not algebraic of degree $\le n$, the equality
$w_n^* (\xi) = n$ holds as soon as $w_n (\xi) = n$. 
Almost all power series $\xi$ in $\F_q ((T^{-1}))$ satisfy 
$w_n^* (\xi) = n$ for every $n \ge 1$. 
\end{theorem}

The first assertion of Theorem \ref{Th:Metric} was stated without proof, and with $w_n^@$ in place 
of $w_n^*$, at the end of \cite{Gu96}. 
It follows immediately from Theorem \ref{Th:wineq} combined with \eqref{eqWir} below, which 
implies that $w_n^* (\xi) \ge n$ holds as soon as we have $w_n (\xi) = n$. 
By a metric result of Sprind\v zuk \cite{Spr69}, stating that almost all 
power series $\xi$ in $\F_q ((T^{-1}))$ satisfy 
$w_n (\xi) = n$ for every $n \ge 1$, this gives the second assertion.

Chen \cite{Chen18} established that, for any $n \ge 1$ and any real number $w \ge n$, the set 
of power series $\xi$ in $\F_q((T^{-1}))$ such that $w_n (\xi) = w$ (\resp $w_n^@ (\xi) = w$) has Hausdorff 
dimension $(n+1)/(w+1)$. In view of Theorem \ref{Th:2.0}, her result also holds for $w_n^*$ in place of $w_n^@$.

As observed by Ooto \cite[Lemma 5.5]{Oo17}, 
it follows quite easily from the theory of continued fractions that $w_1 (\xi) = w_1 (\xi^p)$ for 
every $\xi$ in $\F_q ((T^{-1}))$.
This invariance property extends to the exponents $w_n$ and $w_n^*$. 

\begin{theorem}
\label{Th:powerp} 
Let $\xi$ be in $\F_q ((T^{-1}))$ and $n$ be a positive integer. Then, we have 
$$
w_n (\xi) = w_n (\xi^p), \quad w_n^* (\xi) = w_n^* (\xi^p).
$$
\end{theorem}


It follows from Liouville's inequality (see e.g., \cite[Theorem 5.2]{Oo17}) that, for any $n \ge 1$ and any 
algebraic power series $\xi$ in $\F_q((T^{-1}))$ of degree $d$, we have 
$$
w_n^* (\xi) \le w_n (\xi) \le d-1.
$$
Mahler's example \cite{Mah49} of the root $T^{-1} + T^{-p} + T^{-p^2} + \ldots$ 
of $X^p - X + T^{-1}$ shows that there are algebraic power series $\xi$ in $\F_p((T^{-1}))$ of degree $p$ with 
$w_1 (\xi) = p-1$. 
For further results on Diophantine exponents of approximation of algebraic power series, the 
reader is directed to \cite{Chen13,Tha11,Tha13,Fir13}
and the references given therein.

The present paper is organized as follows.
Further exponents of approximation are defined in Section 2 and (in)equalities between 
them are stated. Auxiliary results are gathered in Section 3, while the next two sections are devoted 
to proofs. Several open questions are listed in Section 6.

Throughout this paper, the notation $\ll$, $\gg$ means that there is an implicit, absolute, 
positive constant.

\section{Uniform exponents and two inequalities between exponents}

A difficulty occurring in the proof of the metric statement of Sprind\v zuk mentioned 
in the previous section is caused by the fact that the polynomials which are very small 
at a given power series could be inseparable. Or, said differently, by the possible existence 
of power series $\xi$ for which $w_n (\xi)$ exceeds $w_n^{{\rm sep}} (\xi)$, where 
$w_n^{{\rm sep}}$ is defined exactly as $w_n$, but with the extra requirement that 
the polynomials $P(X)$ have to be separable. The next result shows that such power series 
do not exist. Before stating it, we define several exponents of uniform approximation.

\begin{definition}
\label{Def:3.1}
Let $\xi$ be in $\F_q ((T^{-1}))$. Let $n \ge 1$ be an integer.
We denote by
$\hw_n (\xi)$ (resp., $\hw_n^{{\rm sep}} (\xi)$) 
the supremum of the real numbers $\hw$ for which there exists 
an integer $H_0$ such that, for every $H > H_0$, there exists a polynomial $P(X)$ 
(resp., a separable polynomial $P(X)$) over $\F_q[T]$ of
degree at most $n$ and height at most $H$ such that 
$$
0 < |P(\xi )| < H^{- \hw}. 
$$
We denote by
$\hw_n^* (\xi)$ the supremum of the real numbers $\hw^*$ for which there exists 
an integer $H_0$ such that, for every $H > H_0$, there exists an
algebraic power series $\alpha$ in $\F_q ((T^{-1}))$ of
degree at most $n$ and height at most $H$ such that 
$$
0 < |\xi - \alpha| < H(\alpha)^{-1} \, H^{-\hw^*}. 
$$
We denote by
$\hw_n^@ (\xi)$ the supremum of the real numbers $\hw^@$ for which there exists 
an integer $H_0$ such that, for every $H > H_0$, there exists an
algebraic power series $\alpha$ in $\F_q ((T^{-1}))$ of
degree at most $n$ and height at most $H$ such that 
$$
0 < |\xi - \alpha| < H(\alpha)^{-1} \, H^{-\hw^@}. 
$$
\end{definition}

For any power series $\xi$ and any $n \ge 1$, we have clearly
$$
\hw_n^{{\rm sep}} (\xi) \le \hw_n  (\xi) 
\quad \hbox{and} \quad
\hw_n^* (\xi) \le \hw_n^@ (\xi). 
$$
The first of these is an equality.

\begin{theorem}
\label{invsep}
Let $\xi$ be in $\F_q ((T^{-1}))$ and $n$ a positive integer. Then, we have 
$$
w_n (\xi) = w_n^{{\rm sep}} (\xi), \quad \hw_n (\xi) = \hw_n^{{\rm sep}} (\xi),
$$
and
$$
w_n (\xi) = w_n(\xi^p), \quad \hw_n (\xi) = \hw_n(\xi^p). 
$$
\end{theorem}

To prove Theorem \ref{Th:Metric}, 
we establish the following inequalities.

\begin{theorem}
\label{WirsUnif}
Let $n \ge 1$ be an integer.  The lower bounds
\begin{equation} \label{eqWir}
w_n^* (\xi) \ge \hw_n^@ (\xi) \ge  \frac{w_n (\xi)}{w_n (\xi) -n +1}  
\end{equation}
and
$$
{w}_n^* (\xi) \ge \frac{\hw_n (\xi)}{\hw_n  (\xi) -n +1} 
$$
hold for any power series $\xi$ which is not
algebraic of degree  $\le n$.
\end{theorem}

For completeness, we define two exponents of simultaneous approximation and 
establish that they are invariant under the map $\xi \mapsto \xi^p$.

Below, the `fractional part' $\Vert \cdot \Vert$ is defined by
$$
\Bigl\Vert \sum_{n=N}^{+ \infty} \, a_n T^{-n} \Bigr\Vert = \Bigl| \sum_{n=1}^{+ \infty} \, a_n T^{-n} \Bigr|,
$$
for every power series $\xi = \sum_{n=N}^{+ \infty} \, a_n T^{-n}$ in $\F_q ((T^{-1}))$. 

\begin{definition}
\label{Def:lambda}
Let $\xi$ be in $\F_q ((T^{-1}))$. Let $n \ge 1$ be an integer.
We denote by
$\lambda_n (\xi)$ the supremum of the real numbers $\lambda$ for which
$$
0 < \max\{ \Vert R(T) \xi \Vert, \ldots , \Vert R(T) \xi^n \Vert \} < q^{-\lambda  \deg(R)}
$$
has infinitely many solutions in polynomials $R(T)$ in $\F_q[T]$. We denote by
$\hla_n (\xi)$ the supremum of the real numbers $\hla$ for which there exists 
an integer $d_0$ such that, for every $d > d_0$, there exists a polynomial $R(T)$ in $\F_q[T]$ of
degree at most $d$ such that 
$$
0 <  \max\{ |R(T) \xi |, \ldots , |R(T) \xi^n|\}  < q^{- \hla d}. 
$$
\end{definition}

Since $\F_q$ is a finite field, requiring `infinitely many solutions in polynomials $R(T)$ in $\F_q[T]$' is equivalent 
to requiring `solutions in polynomials $R(T)$ in $\F_q[T]$ of arbitrarily large degree'.

\begin{proposition} 
\label{lahla}
Let $\xi$ be in $\F_q ((T^{-1}))$ and $n$ a positive integer. Then, we have 
$$
\lambda_n(\xi)=\lambda_n(\xi^p), \quad \hla_n(\xi)=\hla_n(\xi^p).
$$ 
\end{proposition}

\section{Auxiliary results}

\begin{lemma}[Krasner's lemma]  \label{Kras} 
Let $K$ be a complete, algebraically closed field equipped with a non-Archimedean absolute value $| \cdot |$.  
Let $\alpha$ be an algebraic element of $K$ of degree $d$ at least equal to $2$ and separable. 
Let $\alpha = \alpha_1, \alpha_2, \ldots , \alpha_d$ be the conjugates of $\alpha$. 
For any $\beta$ in $K$ satisfying 
$$
|\alpha - \beta| < |\alpha_j - \beta|, \quad 2 \le j \le d,
$$
we have $K(\alpha) \subset K(\beta)$. 
\end{lemma}

\begin{proof}
See e.g. \cite[Section 3.4.2]{BGR84}. 
\end{proof}



\begin{lemma} \label{Gu2}
Let $P (X)$ be a polynomial in $C_{\infty} [X]$ of degree $n \ge 1$ and with leading coefficient $a_n$.
Let $\beta_1, \ldots , \beta_n$ be its roots in $C_{\infty}$. 
Then, for any $\rho > 0$ and any $\xi$ in $C_{\infty}$, we have
$$
\prod_{i=1}^n \max\{|\xi - \beta_i|, \rho\} \ll \gg \frac{H(P)}{|a_n|}
$$
and
$$
\prod_{i=1}^n \min \{|\xi - \beta_i|, \rho\} \ll \gg \frac{|P(\xi)|}{H(P)}. 
$$
\end{lemma}

\begin{proof}
See \cite{Gu96}. 
\end{proof}

\begin{lemma}
\label{lem:estimate}
Let $P(X) = c_m(T) X^m + \ldots + c_1(T) X + c_0 (T)$ be a polynomial in $\F_q[T][X]$ of 
positive degree. Let $\alpha_1,\ldots , \alpha_m$ be its roots in $C_\infty$. 
Let $\xi$ be in $\F_q ((T^{-1}))$. 
Then, for any nonempty subset $S$ of $\{1, \ldots , m\}$, we have 
$$
|c_m(T) \prod_{i\in S} (\xi-\alpha_i)| \le (\max(1,|\xi|))^m H(P).   
$$
\end{lemma}
\begin{proof}
We may assume without loss of generality that $\alpha_1,\ldots ,\alpha_s$ 
are the zeros of $P(X)$ with $|\alpha_i|\ge 1$ and that $|\alpha_j|<1$ for $j>s$.
Let $S_0$ be the set of $i$ in $S$ 
with $|\alpha_i| > |\xi|$.    
Then 
$$
|c_m(T) \prod_{i\in S} ( \xi-\alpha_i)| 
\le |c_m(T)|\prod_{i\in S_0} |\alpha_i|  |\xi|^{|S|-|S_0|}\le (\max(1,|\xi|))^m |c_m(T)\alpha_1\cdots \alpha_s|.
$$
Now
by construction $|\alpha_1\cdots \alpha_s| > |\alpha_{i_1}\cdots \alpha_{i_s}|$ whenever 
$i_1<i_2<\cdots <i_s$ and $i_s\neq s$.  Thus, 
$$
|c_m(T)\alpha_1\cdots \alpha_s| = 
\Bigl|c_m(T) \sum_{i_1<\cdots <i_s} \alpha_{i_1}\cdots \alpha_{i_s} \Bigr| = |c_{m-s}(T)|,
$$ 
and so the result follows.
\end{proof}

\section{Proof of Theorem \ref{invsep}} 

We establish, in this order, that the exponents $w_n$ and $w_n^{{\rm sep}}$ coincide, that 
$w_n$ and $\hw_n$ are invariant under the map $\xi \mapsto \xi^p$, and, finally, 
that the exponents $\hw_n$ and $\hw_n^{{\rm sep}}$ coincide. 

A common key tool for the proofs given in this section is the notion of Cartier operator. 
For a positive integer $j$, let $\Lambda_0,\ldots \Lambda_{p^j-1}:{{\mathbb F}}_q((T^{-1}))\to {{\mathbb F}}_q((T^{-1}))$ 
be the operators uniquely defined by
$$
G(T^{-1}) = \sum_{i=0}^{p^j-1} T^{i} \Lambda_i(G(T^{-1}))^{p^j}, 
$$ 
for $G(T^{-1})$ in ${{\mathbb F}}_q((T^{-1}))$.
Observe that $\Lambda_i(A+B^{p^j} C) = \Lambda_i(A) + B\Lambda_i(C)$ 
for $A,B,C$ in ${{\mathbb F}}_q((T^{-1}))$.
Note also that, for $i=0, \ldots , p^j - 1$, we have
$$
\Lambda_i (T^{-p^j + i}) = T^{-1}, \quad \Lambda_i (T^{p^j + i}) = T.
$$

\begin{proof}[$\bullet$  Proof of the equality $\hw_n=\hw_n^{\rm sep}$.]   

The following lemma implies that the exponents $w_n$ and $w_n^{{\rm sep}}$ coincide.

\begin{lemma} \label{CartOp}
Let $n$ be a positive integer and $\xi$ in $\F_q ((T^{-1}))$. 
Let $w \ge 1$ be a real number and $P (X)$ be in ${{\mathbb F}}_q[T] [X]$ such that $0 < |P(\xi)| \le H(P)^{-w}$.
Then, there exists a separable polynomial $Q(X)$ in ${{\mathbb F}}_q[T] [X]$ such that 
$0 < |Q(\xi)| \le H(Q)^{-w}$. 
\end{lemma}

\begin{proof} 
We start with a polynomial $P(X)$ in ${{\mathbb F}}_q [T] [X]$ 
of degree at most $n$ with 
$$
0< |P(\xi)| = H(P)^{-w}.
$$
Write $P (X) = \sum_{i=0}^n Q_{i}(T) X^i$.  
Let $d$ denote the greatest common divisor of all $i$ such that $Q_{i}$ is nonzero.   
Let $j$ be the nonnegative integer such that $p^j$ divides $d$ but $p^{j+1}$ does not. 
If $j=0$, then $P(X)$ is separable, so we can assume $j \ge 1$. 
Thus we may write
$$
P (X) = \sum_{i\le n/p^j} Q_{p^j i}(T) X^{p^j i}.
$$  
Then, 
$$
P (\xi) = \sum_{i\le n/p^j} Q_{   p^j i}(T) \xi^{p^j i}, 
$$
so, for $0\le s<p^j$, we have
$$
\Lambda_s (P(\xi)) = \sum_{i\le n/p^j} \Lambda_s (Q_{p^j i}(T)) \xi^i.
$$  
Set $n'=\lfloor n/p^j \rfloor$.  
Notice that 
$$
G_{s}(X):=\sum_{i\le n'} \Lambda_s (Q_{p^j i}(T)) X^i
$$
is of degree at most $n'$ and $G_s(\xi) =\Lambda_s (P(\xi))$. 
Thus, 
$$
P(\xi) = \sum_{s=0}^{p^j-1} T^{s} G_{s}(\xi)^{p^j}.
$$ 
Since the $\nu(T^{s})$ are pairwise distinct mod $p^j$ for $s=0,\ldots ,p^j-1$, 
we see that the $\nu( T^{s} G_{s}(\xi)^{p^j})$ 
are pairwise distinct as $s$ ranges from $0$ to $p^j-1$.  
Let
$$
v = \min_{0 \le s \le p^j - 1} \{-s + p^j \nu(G_{s}(\xi))\}.
$$ 
Then
$\nu(G_{i}(\xi)) \ge v /p^j$ for $i=0,\ldots ,p^j-1$ and there exists one $s$ such that
$\nu(G_s(\xi)) = v/p^j +s/p^j$. 	
For this particular $s$ we must have
$$
0 < | G_{s}(\xi) | = q^{-v /p^j - s/p^j} \le q^{-v /p^j } = |P (\xi) |^{1/p^j}.  
$$
Also, by construction, if $B(T)$ is a polynomial of degree $\ell$, then $\Lambda_s (B)$ 
has degree at most $\ell/p^j$ and so $H(G_{s}) \le H(P)^{1/p^{j}}$.  Thus
$$
0 < | G_{s}(\xi) | =  | P_m (\xi) |^{1/p^j}  \le H(P)^{- w /p^j}  
\le H(G_{s})^{- w}.  
$$
If $G_s (X)$ is inseparable, then one repeats the argument to obtain a nonzero 
polynomial of lesser degree small at $\xi$. After at most $\log_p n$ steps we will get a separable polynomial. 
This proves the lemma. 
\end{proof}

\end{proof}

\begin{proof}[$\bullet$ Proof that $w_n$ is invariant under the map $\xi \mapsto \xi^p$.]  
Let $\xi$ and $n$ be as in the theorem. 
Let $P(X)$ be a polynomial of degree at most $n$ in $\F_q[T] [X]$ and define $w$ by 
$|P(\xi)| = H(P)^{-w}$. Write 
$$
P(X) = a_n(T) X^n + \ldots + a_1 (T) X + a_0 (T)
$$
and
$$
Q(X) = a_n(T^p) X^n + \ldots + a_1 (T^p) X + a_0 (T^p). 
$$
Then, we have $P(\xi)^p = Q(\xi^p)$, $H(Q) = H(P)^p$, and 
$$
|Q(\xi^p)|= H(P)^{-pw} = H(Q)^{-w}.
$$
This shows that $w_n(\xi^p) \ge w_n(\xi)$. 

The reverse inequality is more difficult and rests on Lemma \ref{CartOp}. 
Let $P(X)$ be a polynomial of degree at most $n$ in $\F_q[T] [X]$ 
which does not vanish at $\xi^p$, and define $w$ by 
$|P(\xi^p)| = H(P)^{-w}$. Write 
$$
P(X) = a_n(T) X^n + \ldots + a_1 (T) X + a_0 (T)
$$
and
$$
Q(X) = a_n(T) X^{pn} + \ldots + a_1 (T) X^p + a_0 (T)
$$
Then, we have $P(\xi^p) = Q(\xi)$ and $|Q(\xi)| = H(Q)^{-w}$. 
Obviously, $Q(X)$ is not separable. It follows from Lemma \ref{CartOp} that there
exists a polynomial $R(X)$, of degree at most $n$,
such that 
$$
|R(\xi)| \ll H(R)^{-w}.
$$
This shows that $w_n(\xi) \ge w_n(\xi^p)$ and completes the proof of the theorem. 
\end{proof}

In the next proofs, we make use of the following convention.  
Given a nonzero polynomial $P(X)=a_0+a_1X+\cdots +a_m X^m$ in $\mathbb{F}_q[T][X]$
and $i=0, \ldots , p-1$, we let
$\Lambda_i (P)(X)$ denote the polynomial
$$
\sum_{j=0}^m \Lambda_i (a_j) X^j.
$$


\begin{proof}[$\bullet$  Proof that $\hw_n$ is invariant under the map $\xi \mapsto \xi^p$.] 
Let $\varepsilon>0$.  
By assumption, for any sufficiently large $H$, there is some polynomial 
$P(X)=a_0+a_1 X+ \cdots +a_m X^m$ of degree $m$ at most $n$ and height at most $H^{1/p}$ such that
$$
0<|P(\xi)|<H^{-\hw_n(\xi)/p+\varepsilon/p}.
$$  
Set $Q(X)=a_0^p+a_1^p X+\cdots + a_m^p X^m$.  
Then $Q(X)$ has degree at most $n$ and height at most $H$, and, by construction, it satisfies 
$$
0<|Q(\xi^p)|<H^{-\hw_n(\xi)+\varepsilon}.
$$ 
It follows that $\hw_n(\xi^p)\ge \hw_n(\xi) - \varepsilon$ for every $\varepsilon>0$ and so we get the inequality
$$
\hw_n(\xi^p)\ge \hw_n(\xi).
$$

We now show the reverse inequality.  
By assumption, for any sufficiently large $H$, there is some polynomial 
$Q(X)=a_0+a_1 X +\cdots + a_m X^m$ 
of degree $m$ at most $n$ and height at most $H^p$ such that
$$
0<|Q(\xi^p)| < (H^p)^{-\hw_n(\xi^p)+\varepsilon}.
$$  
Then for each $i$ in $\{0,\ldots ,p-1\}$ we define
$$
Q_i(X)=\sum_{j=0}^m \Lambda_i  (a_j) X^j.
$$
By construction, we have $H(Q_i) \le H$ for $i=0,\ldots , p-1$.
Also we have $Q_i(\xi) = \Lambda_i (Q)(\xi)$ for $i=0,\ldots , p-1$ and so 
$Q(\xi^p) = \sum_{j=0}^{p-1} T^j Q_i(\xi)^p$.  Since the valuations are distinct, we have
$$
|Q_i(\xi)|< H^{-\hw_n(\xi^p) +\varepsilon},
$$ 
for $i=0,\ldots , p-1$.  
Since $Q(X)$ is nonzero, there is some $k$ in $\{0,\ldots ,p-1\}$ such that $Q_k(\xi)\neq 0$ and we see
$$
0<|Q_k(\xi)|<H^{-\hw_n(\xi^p) +\varepsilon}.
$$
It follows that $\hw_n(\xi)\ge \hw_n(\xi^p)$, giving us the reverse inequality.
\end{proof}

The next lemma is used in the proof that the uniform 
exponents $\hw_n$ and $\hw_n^{\rm sep}$ coincide. 
We let $\log_p$ denote the logarithm in base $p$.

\begin{lemma} 
\label{lem:pr}
Let $\xi\in \mathbb{F}_p ((T^{-1}))$ and let $P(X)=c_0 + c_1 X+\cdots + c_n X^{n}\in \mathbb{F}_q[T][X]$ 
be a non-constant polynomial that is a product of irreducible inseparable polynomials such that $P(\xi)$ is nonzero.
Then, there exist an integer $r$ with $0 \le r \le \log_p(n)$ and a polynomial $P_0(X)$ such that the following hold:
\begin{enumerate}
\item $P_0(X)$ has a non-trivial separable factor;
\item $p^r {\rm deg}(P_0) \le {\rm deg}(P)$;
\item $0<|P_0(\xi)|^{p^r} < q^{p^r-1} |P(\xi)|$;
\item $H(P_0)^{p^r} \le H(P)$.
\end{enumerate}
\end{lemma}
\begin{proof}
Suppose that this is not the case.  Then there must be some smallest $n$ for which it is not true. 
Then since $P(X)$ is a product of irreducible inseparable polynomials $n=pm$ for some $m$.  
Then $P(X) = Q(X^p)$ for some polynomial $Q$ of degree $m$.  Observe that 
$$
P(\xi) = \sum_{j=0}^{p-1} T^j (\Lambda_j (Q)(\xi))^p. 
$$ 
For $j = 0, \ldots , p-1$ such that $\Lambda_j(Q)(\xi)$ is nonzero, write
$$
\hbox{$\Lambda_j (Q)(\xi) = c_j T^{-a_j} +$ larger powers of $T^{-1}$},
$$
with $c_j\neq 0$. 
Then, we have
$$ 
\hbox{$T^j (\Lambda_j (Q)(\xi))^p = c_j^p T^{-pa_j+j} +$ larger powers of $T^{-1}$}.
$$
Now there must be some unique $j_0$ such that $pa_{j_0}-j_0$ 
is minimal among all $pa_j-j$ (and it must be finite), thus
$$
|P(\xi)| = q^{-(pa_{j_0}-j_0)}.
$$
Then, the polynomial $A(X):= \Lambda_{j_0} (Q)$ has the property that
$$
|A(\xi)| = q^{-a_{j_0}}.
$$
To summarize, we have:

\begin{enumerate}
\item[(a)] $p{\rm deg}(A) \le {\rm deg}(P)$;
\item[(b)] $0<|A(\xi)|^{p} \le  |P(\xi)|$;
\item[(c)] $H(A)^{p} \le H(P)$.
\end{enumerate}
By construction, ${\rm deg}(A)<n$ and so by minimality, there is some $r\le \log_p(m)$ 
and a polynomial $Q(X)$ such that:
\begin{enumerate}
\item[(d)] $p^r {\rm deg}(Q) \le {\rm deg}(A)$;
\item[(e)] $0<|Q(\xi)|^{p^r} \le |A(\xi)|$;
\item[(f)] $H(Q)^{p^r} \le H(A)$;
\item[(g)] $Q$ has a non-trivial separable factor.
\end{enumerate}
Then by construction
$p^{r+1}{\rm deg Q}\le {\rm deg}(P)$, $0<|Q(\xi)|^{p^{r+1}} < q^{p^{r+1}-1} |P(\xi)|$ 
and $H(Q)^{p^{r+1}}\le H(P)$.  Furthermore, $r+1\le 1+\log_p(m)\le \log_p(n)$ and so we get the desired result.
\end{proof}

\begin{proof}[$\bullet$  Proof of the equality $\hw_n=\hw_n^{\rm sep}$.] 
It is clear that $\hw_n(\xi)\ge \hw_n^{\rm sep}(\xi)$.  
We now show the reverse inequality.  Let $\varepsilon>0$.  Then there is some $H_0$ 
such that for every $H>H_0$ there is a polynomial $P(X)$ of degree at most $n$ and height at most $H$ such that
$$
0<|P(\xi)|<H^{-\hw_n(\xi) +\varepsilon}.
$$
We take the infimum over all $d\le n$ for which there is some positive constant $C$ 
such that for every $H>H_0$ there is a polynomial $A(X)B(X)$ with $A(X)$ separable 
and $B(X)$ a polynomial of degree at most $d$ that is a product of irreducible inseparable polynomials with
$$
0<|A(\xi)B(\xi)| < C\cdot H^{-\hw_n(\xi) +\varepsilon}.
$$
Then by assumption, $d$ must be positive and since the polynomial $B(X)$ 
is a product of inseparable irreducible polynomials, we see that $p$ divides $d$.  

Let $H>H_0$.  Then there is a fixed constant $C>0$ that does not depend on $H$ 
such that there are polynomials $A(X)$ and $B(X)$ with $A$ separable and $B$ 
a polynomial of degree at most $d$ that is a product of irreducible separable polynomials with 
$$
0<|A(\xi)B(\xi)| < C\cdot H^{-\hw_n(\xi) +\varepsilon}.
$$  
Then by Lemma \ref{lem:pr} there is some $r\le \log_p(n)$ 
and a polynomial $B_0(X)$ with a non-trivial separable factor such that 
${\rm deg}(B)=p^r {\rm deg}(B_0)$ and
$$
0<|B_0(\xi)| \le  |B(\xi)|
$$ 
and $H(B_0^{p^r})<H(B)$.  Thus, the polynomial 
$$
A(X)B_0(X)^{p^r}
$$ 
has degree at most $n$ and height at most $H$ and 
$$
0<|A(\xi)B_0(\xi)^{p^r}| < C q^{p^r-1} H^{-\hw_n(\xi) +\varepsilon}.
$$ 
By assumption, we can write $B_0(X)=C(X)D(X)$ with $C(X)$ non-constant and separable 
and $D(X)$ a product of irreducible inseparable polynomials.
Then we have 
$$
{\rm deg}(D(X)^{p^r})\le {\rm deg}(B)-p^r <d.
$$  
But this contradicts the minimality of $d$ and so we see that $d$ must be zero and so we get
$\hw_n^{\rm sep}(\xi)\ge \hw_n(\xi)-\varepsilon$.  Since $\varepsilon>0$ is arbitrary, we get the desired result.
\end{proof}

\begin{proof}[Proof of Proposition \ref{lahla}]
Observe that, for $j \ge 1$, 
the equality $(R(T) \xi^j)^p = R(T^p)\xi^{p j}$ immediately yields $\lambda_n (\xi^p)\ge \lambda_n(\xi)$
and $\hla_n (\xi^p)\ge \hla_n(\xi)$, for $n \ge 1$. 

Take $\lambda$ with $0 < \lambda < \lambda_n(\xi^p)$. 
Then, there is an infinite set $\mathcal{S}$ of polynomials $R(T)$ such that
$$
0 < \max\{ \Vert R(T) \xi^p \Vert, \ldots , \Vert R(T) \xi^{pn} \Vert \} < q^{-\lambda  \deg(R)}.
$$
By replacing $\mathcal{S}$ with a well-chosen infinite subset, 
we may assume that there is a fixed $j$ in $\{0,1,\ldots, p-1\}$ 
such that the degree of every polynomial in $\mathcal{S}$ is congruent to $j$ modulo $p$.  
Then for $R(T)$ in $\mathcal{S}$ and $i$ in $\{1,\ldots ,n\}$, we apply the $j$-th Cartier operator $\Lambda_j $ to 
$R(T)\xi^{pi}$ and we have $|\Lambda_j  (R(T))\xi^{i}| < q^{-\lambda \deg(R)/p}$
We let $Q(T)=\Lambda_j  (R(T))$.  Then the degree of $Q$ is $(\deg(R)-j)/p$ and so we see
$$
\left| Q(T) \xi^i\right| < q^{-\lambda (p\deg(Q) +j)/p} \le q^{-\lambda \deg(Q)}, 
$$ 
for $i=1,\ldots ,n$.  Since the degrees of the elements of 
$\Lambda_j  (\mathcal{S})$ are arbitrarily large, we deduce that $\lambda_n(\xi) \ge \lambda$.
Consequently, we have established that $\lambda_n(\xi) \ge \lambda_n (\xi^p)$.

Take $\hla$ with $0 < \hla < \hla_n(\xi^p)$. 
For any sufficiently large integer $d$, there is a polynomial $R(T)$ of degree at most $p d$ such that
$$
0 < \max\{ \Vert R(T) \xi^p \Vert, \ldots , \Vert R(T) \xi^{pn} \Vert \} < q^{-\hla  p d}.
$$
Let $j$ be in $\{0,1,\ldots, p-1\}$ 
such that the degree of $R(T)$ is congruent to $j$ modulo $p$.  Apply the $j$-th Cartier operator to 
$R(T)\xi^{pi}$ and let $Q(T)=\Lambda_j  (R(T))$.  Then the degree of $Q$ is at most equal to $d$ and so we see
$$
\left| Q(T) \xi^i\right| < q^{- \hla  p d /p} \le q^{- \hla d}, 
$$ 
for $i=1,\ldots ,n$. This shows that $\hla_n(\xi) \ge \hla$. Thus, we obtain $\hla_n (\xi) \ge \hla_n (\xi^p)$.
\end{proof}

\section{Proofs of Theorems \ref{Th:2.0}, \ref{Th:wineq}, \ref{Th:powerp}, and \ref{WirsUnif}}

By adapting the proof of Wirsing \cite{Wir61} to the 
power series setting, Guntermann \cite[Satz 1]{Gu96} established that, for every $n \ge 1$
and every $\xi$ in $\F_q ((T^{-1}))$ not algebraic of degree $\le n$, we have
$$
w_n^@ (\xi) \ge \frac{n+1}{2}.
$$
Actually, it is easily seen that instead of starting her proof with polynomials given by Mahler's analogue \cite{Mah41,Spr69}
of Minkowski's theorem, she could have, like Wirsing, started with polynomials $P[X]$ satisfying
$$
0 < |P(\xi)| < H(P)^{-w_n (\xi) + \eps},
$$
where $\eps$ is an arbitrarily small positive real number. By doing this, one gets the stronger assertion
\begin{equation} \label{GuW}
w_n^@ (\xi) \ge \frac{w_n (\xi) +1}{2},
\end{equation}
which is crucial for proving Theorem \ref{Th:2.0}. 
Note that Guntermann \cite{Gu96} did not obtain any lower bound for $w_n^* (\xi)$, except when $n=2$.

\begin{proof}[Proof of Theorem \ref{Th:2.0}]
Set $w = w_n (\xi)$, $w^@ = w_n^@ (\xi)$, and $w^* = w_n^* (\xi)$. 
Suppose that $w^@  > w^*$ and pick $\varepsilon$ in $(0, 1/3)$ such that $w^@ > w^* + 2\varepsilon$.  
Then, there are infinitely many $\alpha$  in $C_{\infty}$ algebraic of degree at most $n$ such that
$$
|\xi - \alpha| < H(\alpha)^{-1-w^@ + \varepsilon}.
$$
Let $P_\alpha (X)$ denote the minimal polynomial of $\alpha$ over $\F_q[T]$.  Then $H(P_\alpha) = H(\alpha)$.  
We let $\alpha=\alpha_1,\ldots ,\alpha_m$ denote the roots of $P_\alpha(X)$ 
(with multiplicities), where $m={\rm deg}(P_\alpha)\le n$. 
We may assume that $|\xi-\alpha_1|\le \cdots \le |\xi-\alpha_m|$.  Let $r$ be the largest integer such that 
$$
|\xi-\alpha_1|=\cdots = |\xi-\alpha_r|. 
$$

If $r=1$ for infinitely many $\alpha$ as above, then $P_\alpha(X)$ is separable over $\mathbb{F}_q((T))$,  
and we conclude from Krasner's Lemma \ref{Kras} that $\alpha_1$ lies in $\mathbb{F}_q((T^{-1}))$. 
For $H(\alpha)$ large enough, we then get
$$
H(\alpha)^{-1-w_n^*-\varepsilon} < |\xi - \alpha| < H(\alpha)^{-1-w^@ + \varepsilon}, 
$$
thus $w^@ \le w^* + 2\varepsilon$, a contradiction. 

Thus, we have $r \ge 2$.
Observe that
$|P_\alpha(\xi)| > H(\alpha)^{-w -\varepsilon}$ if $H(P_\alpha)$ is large enough.  
On the other hand, with $c_\alpha (T)$ being the leading coefficient of $P_\alpha (X)$, we get 
\begin{align*} 
|P_\alpha(\xi)| &= \Bigl| c_\alpha(T) (\xi-\alpha_1)\cdots (\xi -\alpha_r) \prod_{i=r+1}^m (\xi-\alpha_i) \Bigr| \\
& =  | \xi-\alpha|^r  \cdot \Bigl|c_\alpha(T) \prod_{i=r+1}^{m} (\xi-\alpha_i) \Bigr| \\
&< (\max\{1,|\xi|\})^n \cdot H(\alpha)^{1-r (1+w^@-\varepsilon)},
\end{align*}
where the last step follows from Lemma \ref{lem:estimate}.

By \eqref{GuW} we have $w^@ \ge (w+1)/2$, thus we get 
$$
H(\alpha)^{-w -\varepsilon} \ll |P_\alpha(\xi)| \ll H(\alpha)^{1-r (1+w^@-\varepsilon)} 
\ll H(\alpha)^{1-r(1+(w+1)/2)  + r\varepsilon}.
$$
This then gives 
$$
w +\varepsilon \ge -1 + r+ \frac{r(w+1)}{2} - r\varepsilon,
$$ 
and since $r\ge 2$ we deduce 
$$
w +\varepsilon \ge - 1+w +1 +r - r\varepsilon,
$$
which is absurd. Since $\eps$ can be taken arbitrarily small, we deduce that $w_n^@ (\xi) \le w_n^* (\xi)$. 
As the reverse inequality immediately follows from the definitions of $w_n^@$ and $w_n^*$, the proof 
is complete. 
\end{proof}

We are ready to complete the proof of Theorem \ref{Th:wineq}. 

\begin{proof}[Proof of Theorem \ref{Th:wineq}]
Let $\xi$ be in $\F_q ((T^{-1}))$ and $n$ be a positive integer. The 
inequality $w_n^* (\xi) \le w_n (\xi)$ is clear. 
Let $\eps$ be a positive real number. By Lemma \ref{CartOp}, there exist separable 
polynomials $P(X)$ in $\F_q[T] [X]$ of arbitrarily large height such that
$$
0 < |P(\xi)| < H(P)^{- w_n (\xi) +\eps}.
$$
Then, the (classical) argument given at the beginning 
of the proof of \cite[Lemma 5.4]{Oo17} yields the existence of a root $\alpha$ of $P(X)$ such that 
$$
0 < |\xi - \alpha| \le |P(\xi)| \, H(P)^{n - 2}. 
$$
Thus, we get the inequality
$$
w_n^@ (\xi) \ge w_n (\xi) - n + 1,
$$
and we conclude by applying Theorem \ref{Th:2.0} which asserts that $w_n^@ (\xi) = w_n^* (\xi)$. 
\end{proof}

\begin{proof}[Proof of Theorem \ref{Th:powerp}]
Let $\xi$ and $n$ be as in the theorem. 
In view of Theorem \ref{invsep}, it only remains for us to prove that $w_n^* (\xi) = w_n^* (\xi^p)$. 
Let $\alpha$ be in $\F_q((T^{-1}))$ algebraic of degree at most $n$ and define $w$
by $|\xi - \alpha| = H(\alpha)^{-w-1}$. Then, we have
$$
|\xi^p - \alpha^p| = H(\alpha)^{-p(w+1)} \quad \hbox{and} \quad H(\alpha^p) \le H(\alpha).
$$
Consequently, we get $w_n^* (\xi^p) \ge w_n^* (\xi)$. 

Now, we prove the reverse inequality. Set $w = w_n^* (\xi^p)$. 
Let $\eps > 0$ and $\alpha$ be in $\F_q((T^{-1}))$ algebraic of degree at most $n$ 
such that $|\xi^p - \alpha| < H(\alpha)^{-w-1 + \eps}$. 
We look at $\xi^p$ as an element of the field
$\F_q((U^{-1}))$, where $U = T^p$. Note that $\alpha$ is in the algebraic closure of $\F_q((U^{-1}))$. 
Consequently, in the field $\F_q((U^{-1}))$, we have $w_n^@ (\xi^p) = w$. 
By Theorem \ref{Th:2.0}, we obtain that, in the field $\F_q((U^{-1}))$, we have $w_n^* (\xi^p) = w$, thus 
there are $\beta$ in $\F_q((U^{-1}))$ algebraic of degree at most $n$ of arbitrarily large 
height such that $|\xi^p - \beta| < H(\beta)^{-w-1 + \eps}$. But these $\beta$ are of the form $\beta = \gamma^p$,
with $\gamma$ in $\F_q((T^{-1}))$, so we get 
$$
|\xi - \gamma| < H(\gamma)^{-w-1 + \eps}, 
$$
and we deduce that $w_n^* (\xi) \ge w_n^* (\xi^p)$.
\end{proof}

\begin{proof}[Proof of Theorem \ref{WirsUnif}]
We obtain \eqref{eqWir}  by taking Wirsing's
argumentation \cite{Wir61}. 
Let $n \ge 2$ be an integer and let $\xi$ be
a power series which is either transcendental, 
or algebraic of degree $> n$.
Let $\eps > 0$ and set $w = w_n(\xi) (1 + \eps)^2$. 
Let $i_1, \ldots, i_n$ be distinct integers in $\{0, \ldots , n\}$ such that 
$\nu (\xi) \not= i_j$ for $j = 1, \ldots , n$. 
By Mahler's analog \cite{Mah41,Spr69} of Minkowski's theorem, 
there exist a
constant $c$ and, for any positive real number $H$,
a nonzero polynomial
$P(X)$ of degree at most $n$ such that
$$
|P(\xi)| \le H^{-w}, \quad |P(T^{i_1})|, \ldots, |P(T^{i_{n-1}})| \le H,  \quad
{\rm and} \quad |P(T^{i_n})| \le c H^{w-n+1}.  
$$
The definitions of  $w_n(\xi)$ and $w$ show
that $H(P) \gg H^{1 + \eps}$.
It follows from Lemma \ref{Gu2} 
that $P(X)$ has some root in a small neighbourhood of each of the points 
$\xi$, $T^{i_1}, \ldots, T^{i_{n-1}}$. Denoting by $\alpha$
the closest root to $\xi$, we get
$$
|\xi - \alpha| \gg \ll  \frac{|P(\xi)|}{H(P)} \ll H(P)^{-1} \,
(H^{w-n+1})^{-w/(w-n+1)} \quad
$$
and
$$
H(P) \ll H^{w-n+1}.
$$
Since all of this holds for any sufficiently large $H$, we deduce that 
$\hw_n^@ (\xi) \ge
w/(w-n+1)$.  Selecting now $\eps$ arbitrarily close to $0$, we obtain
the first assertion. 

Now, we establish the second assertion. 
Since $\hw_n (\xi) \ge n$,
there is nothing to prove if $w_n^*(\xi) \ge n$.
Otherwise, let $A >2$ be a real number with $w_n^*(\xi) < A - 1 < n$.
Thus, we have $|\xi - \alpha| \ge H(\alpha)^{-A}$ for all algebraic power
series $\alpha$ of degree  $\le n$ and sufficiently large height.
We make use of an idea of Bernik and Tishchenko; see also
\cite[Section 3.4]{BuLiv}. Let $\eps > 0$ be given. 
We may assume that $|\xi| \le 1$. 
Again, by Mahler's analog \cite{Mah41,Spr69} of Minkowski's theorem, there exist a
constant $c$ and, for any positive real number $H$,
a nonzero polynomial
$P(X) = a_n X^n + \ldots + a_1 X + a_0$ of degree at most $n$ such that
$$
|a_1|\le H^{1 + \eps}, \quad |a_2|, \ldots , |a_n| \le H, \quad 
|P(\xi)| \le c H^{-n-\eps}.
$$
If $P(X)$ is a product of irreducible inseparable factors, then $a_1 = 0$ and $H(P) \ll H$. 
Assume now that $P(X)$ has a separable factor. 
Let $\alpha$ in $C_\infty$ be the closest root of $P(X)$ to $\xi$. 
If $|a_1| > H$, then we deduce from 
$|n a_n \xi^{n-1} + \ldots + 2 a_2 \xi| \le H$ that 
$|P'(\xi)| = |a_1| \gg H(P)$. Thus, we get 
$$
|P(\xi)| \ge |\xi - \alpha| \cdot |P'(\xi)| \gg H(\alpha)^{1 - A}
$$
and
$$
|P(\xi)| \le H(P)^{-(n+ \eps)/(1 + \eps)}, 
$$
which also holds if $a_1 = 0$. 
Consequently, if $A- 1 \le (n+ \eps)/(1 + \eps)$, that is, if
$$
\eps < \frac{n+1 - A}{A - 2},
$$
then we get a contradiction if $H$ is large enough. 
We conclude that, for any $\eps < (n+1 - A)/(A-2)$ and any sufficiently large $H$, 
there exists a polynomial $P(X)$ of height  $\le H$
and degree $\le n$ satisfying
$|P(\xi)| \le H^{-n-\eps}$. 
Consequently, we have
$\hw_n (\xi) \ge n + \eps$, and thus 
$\hw_n (\xi) \ge n + (n+1 -  A)/(A-2)$.
We obtain the desired inequality by letting $A$ tend to $1 + w_n^*(\xi)$.  
\end{proof}

\section{Further problems}

Despite some effort, we did not succeed to solve the following problem.

\begin{problem} 
Let $n$ be a positive integer and $\xi$ in $\Q_p$. Prove that
$$
\hw_n^* (\xi) = \hw_n^@(\xi) = \hw_n^* (\xi^p). 
$$
\end{problem}

For $n \ge 2$, Ooto \cite{Oo17} proved the existence of $\xi$ in $\F_q ((T^{-1}))$ for which $w_n^* (\xi) < w_n (\xi)$. 
His strategy, inspired by \cite{Bu12}, 
was to use continued fractions to construct power series $\xi$ with $w_2^* (\xi) < w_2 (\xi)$ 
and $w_2^* (\xi)$ sufficiently large to ensure that, for small (in terms of $w_2^* (\xi)$) values of $n$, we have 
$$
w_2^* (\xi) = w_3^* (\xi) = \ldots = w_n^* (\xi), \quad
w_2 (\xi) = w_3 (\xi) = \ldots = w_n (\xi). 
$$
Very recently, Ayadi and Ooto \cite{AyOo20} answered a question of Ooto \cite[Problem 2.2]{Oo18} by
proving, for given $n \ge 2$ and $q \ge 4$, the existence of algebraic power series 
$\xi$ in $\F_q ((T^{-1}))$ for which $w_n^* (\xi) < w_n (\xi)$.

\begin{problem} 
Do there exist power series $\xi$ in $\F_q ((T^{-1}))$ such that
$$
w_n^* (\xi) < w_n (\xi), \quad \hbox{for infinitely many $n$?}
$$
\end{problem}

The formulation of the next problem is close to that of \cite[Problem 2.4]{Oo18}. 

\begin{problem} 
Let $\xi$ be an algebraic power series in $\F_q ((T^{-1}))$ and $n$ a positive integer. 
Is $w_1 (\xi)$ always rational? Are $w_n(\xi), w_n^* (\xi),$ and 
$\lambda_n (\xi)$ always algebraic numbers? 
\end{problem}

No results are known on uniform exponents of algebraic power series in $\F_q ((T^{-1}))$. 

\begin{problem} 
Let $\xi$ be an algebraic power series in $\F_q ((T^{-1}))$ and $n$ a positive integer. Do we have 
$$
\hw_n (\xi) = \hw_n^* (\xi) = n ?
$$
\end{problem}

In the real case, there are many of relations between the six exponents
$w_n$, $w_n^*$, $\lambda_n$, $\hw_n$, $\hw_n^*$, $\hla_n$, see e.g. 
the survey \cite{BuDurham}.
We believe that most of the proofs 
can be adapted to the power series setting.

\end{document}